%% file: note-lattice-face.tex
\documentclass{amsart}
\include{headerstemplate}

\address{Department of Mathematics, One Shields Avenue, University of California, Davis 95616.}
\email{fuliu@math.ucdavis.edu}

\begin{document}
\title{A note on lattice-face polytopes and their Ehrhart polynomials}
\author{Fu Liu}
\begin{abstract}
We give a new definition of lattice-face polytopes by removing an unnecessary restriction in \cite{lattice-face}, and show that with the new definition, the Ehrhart polynomial of a lattice-face polytope still has the property that each coefficient is the normalized volume of a projection of the original polytope. Furthermore, we show that the new family of lattice-face polytopes contains all possible combinatorial types of rational polytopes.
\end{abstract}

\maketitle

\section{Introduction}

A {\it convex polytope} is a convex hull of a finite set of points.
We often omit ``convex'' and just write polytope. The {\it face poset} of a polytope $P$ is the set of all faces of $P$ ordered by inclusion. We say two polytopes have the {\it same combinatorial type} if they have the same face poset. 

The $d$-dimensional {\it lattice} $\Z^d = \{\x = (x_1, \dots, x_d) \ |
\  x_i \in \Z\}$ is the collection of all points with integer
coordinates in $\R^d.$ Any point in a lattice is called a {\it
lattice point}. For any polytope $P$ and positive integer $m \in \N,$ we denote by $i(m,P)$ the
number of lattice points in $m P,$ where $m P = \{ m x | x \in P \}$
is the {\it $m$th dilated polytope} of $P.$

An {\it integral} or {\it lattice} polytope is a convex polytope
whose vertices are all lattice points. A {\it rational} polytope is a convex polytope whose vertices are in $\Q^d.$ Eug\`{e}ne Ehrhart
\cite{Ehrhart} discovered that for any $d$-dimensional rational
polytope, $i(P,m)$ is a quasi-polynomial of
degree $d$ in $m,$ whose period divides the least common multiple of
the denominators of the coordinates of the vertices of $P.$ (See \cite{ec1} for a definition of quasi-polynomials. We do not include it here because it is irrelevant to this paper.) In
particular, if $P$ is an integral polytope, $i(P,m)$ is a
polynomial. Thus, we
call $i(P,m)$ the {\it Ehrhart polynomial} of $P$ when $P$ is an integral
polytope. See \cite{Barvinok, BeckRobins} for further references
to the literature of lattice point counting. Although Ehrhart's
theory was developed in the 1960's, we still do not have a very good understanding of
the coefficients of Ehrhart polynomials for general polytopes except
that the leading, second and last coefficients of $i(P,m)$ are the
normalized volume of $P$, one half of the normalized volume of the
boundary of $P$ and $1,$ respectively. 

In \cite{cyclic}, the author showed that for any $d$-dimensional integral cyclic polytope $P,$ we have that \begin{equation}\label{ques}
i(P, m) =  \vol(mP) + i(\pi(P), m) = \sum_{k=0}^d
\vol_k(\pi^{(d-k)}(P)) m^k,
\end{equation}
where $\pi^{(k)}: \R^d \to \R^{d-k}$ is the map which ignores the
last $k$ coordinates of a point. In \cite{lattice-face}, the author generalized the family of integral cyclic polytope to a bigger family of integral polytopes, {\it
lattice-face} polytopes, and showed that their
Ehrhart polynomials also satisfy (\ref{ques}).

One question that has been asked often is: {\it how big is the family of lattice-face polytopes?} The motivation of this paper is to answer this question. We examine the definition of lattice-face polytopes given in \cite{lattice-face}, and notice there is a unnecessary restriction. After removing this restriction, we give a new definition of lattice-face polytopes (Definition \ref{dflf1}). With this new definition, we have two main results of this paper:

\begin{thm}\label{main1}
For any lattice-face polytope $P$, the Ehrhart polynomial of $P$ satisfies \eqref{ques}. In other words, the coefficient of $m^k$ in $i(P,m)$ is the normalized volume of the projection $\pi^{(d-k)}(P)$.
\end{thm}

\begin{thm}\label{main2}
The family of lattice-face polytopes contains all combinatorial types of rational polytopes. More specifically, given any full-dimensional rational polytope $P,$ there exists an invertible linear transformation $\phi$, such that $\phi(P)$ is a lattice-face polytope.
\end{thm}

We remark that the proof given in \cite{lattice-face} for the lattice-face polytopes under the old definition would work for proving Theorem \ref{main1}. However, we use a slightly different approach in this paper. We will prove Theorem \ref{main1} by using Proposition \ref{tri2form}, which provides an even bigger family of polytopes whose Ehrhart polynomials satisfy \eqref{ques}. See Examples \ref{tri1} and \ref{tri2} for polytopes that are not lattice-face polytopes but are covered by Proposition \ref{tri2form}.

It follows from Theorem \ref{main1} and Theorem \ref{main2} that for any rational polytope $P,$ one can apply a linear transformation $\phi$ to $P,$ such that $i(\phi(P),m)$ is a polynomial having the same simple form as (\ref{ques}). Thus, all the coefficients of $i(\phi(P), m)$ have a geometric meaning and are positive. 
This result provides a possible method to prove positivity conjectures on coefficients of the Ehrhart polynomials of integral polytopes. Namely, given an integral polytope $P,$ if among all the piecewise linear transformations $\phi$ such that $\phi(P)$ is a lattice-face polytope, one can find one that preserves the lattice, then one can conclude that the coefficients of $i(P,m)$ are all positive.

\section{Preliminaries}

We first give some definitions and notation, most of which follows
\cite{lattice-face}, and also present relevant results obtained in \cite{lattice-face}.

All polytopes we will consider are full-dimensional unless otherwise noted, so for any
convex polytope $P,$ we denote by $d$ both the dimension of the
ambient space $\R^d$ and the dimension of $P.$ We call a
$d$-dimensional polytope a $d$-polytope. 
We denote by $\partial P$ the boundary.
A {\it $d$-simplex} is a polytope given as the convex hull of $d+1$
affinely independent points in $\R^d.$

For any set $S,$ we denote by $\conv(S)$ the convex hull of all
the points in $S,$ and by $\aff(S)$ the affine hull of all the points in $S.$

Recall that the projection $\pi: \R^d
\to \R^{d-1}$ is the map that forgets the last coordinate. For any
set $S \subset \R^d$ and any point $y \in \R^{d-1},$ let $\rho(y, S)
= \pi^{-1}(y) \cap S$ be the intersection of $S$ with the inverse
image of $y$ under $\pi.$ If $S$ is bounded, let $p(y, S)$ and $n(y, S)$ be the point
in $\rho(y,S)$ with the largest and smallest last coordinate,
respectively. If $\rho(y,S)$ is the empty set, i.e., $y \not\in
\pi(S),$ then let $p(y, S)$ and $n(y,S)$ be empty sets as well.
Clearly, if $S$ is a $d$-polytope, $p(y, S)$ and $n(y, S)$ are on
the boundary of $S.$ Also, we let $\rho^+(y,S) = \rho(y,S)\setminus
n(y,S),$ and for any $T \subset \R^{d-1},$ we let $\rho^+(T,S) = \bigcup_{y
\in T} \rho^+(y,S).$

\begin{defn}\label{dfbdry}
Define $PB(P) = \bigcup_{y \in \pi(P)} p(y,P)$ to be the {\it
positive boundary} of $P;$ $NB(P) = \bigcup_{y \in \pi(P)} n(y,P)$ to
be the {\it negative boundary} of $P$ and $\Omega(P) = P \setminus
NB(P) = \rho^+(\pi(P), P) = \bigcup_{y \in \pi(P)} \rho^+(y,P)$ to be
the {\it nonnegative part} of $P.$
\end{defn}
\begin{defn}
For any facet $F$ of $P,$ if $F$ has an interior point in the
positive boundary of $P,$  then we call $F$ a {\it positive facet}
of $P$ and define the sign of $F$ to be $+1: \sgn(F) = + 1.$ Similarly,
we can define the {\it negative facets} of $P$ with associated sign
$-1.$  
\end{defn}
It's easy to see that $F \subset PB(P)$ if $F$ is a positive facet
and $F \subset NB(P)$ if $F$ is a negative facet.

We write $P = \bigsqcup_{i=1}^k P_i$ if $P = \bigcup_{i=1}^k P_i$
and for any $i \neq j$, $P_i \cap P_j$ is contained in their
boundaries. If $F_1, F_2, \dots, F_{\ell}$ are all the positive
facets of $P$ and $F_{\ell +1}, \dots, F_k$ are all the negative
facets of $P,$ then $$\pi(P) = \bigsqcup_{i=1}^{\ell} \pi(F_i) =
\bigsqcup_{i=\ell+1}^k \pi(F_i).$$

 Because the usual set union
and set minus operation do not count the number of occurrences of an
element, which is important in our paper, from now on we will
consider any polytopes or sets as {\it multisets} which allow {\it
negative multiplicities.} In other words, we consider any element of
a multiset as a pair $(\x, m),$ where $m$ is the multiplicity of
element $\x.$ Then for any multisets $M_1, M_2$ and any integers
$m,n$ and $i,$ we define the following operators:
\begin{alist}
\itm Scalar product: $i M_1 = i \cdot M_1 = \{ (\x, i m) \ | \ (\x,
m) \in M_1\}.$ \itm Addition: $M_1 \oplus M_2 = \{ (\x, m + n) \ | \
(\x, m) \in M_1, (\x, n) \in M_2 \}.$ \itm Subtraction: $M_1 \ominus
M_2 = M_1 \oplus ((-1) \cdot M_2).$
\end{alist}



Let $P$ be a convex polytope. For any $y$ an interior point of
$\pi(P),$ since $\pi$ is a continuous open map, the inverse image of
$y$ contains an interior point of $P.$ Thus $\pi^{-1}(y)$ intersects
the boundary of $P$ exactly twice. For any $y$ a boundary point of
$\pi(P),$ again because $\pi$ is an open map, we have that $\rho(y,
P) \subset \partial P,$ so $\rho(y,P) = \pi^{-1}(y) \cap \partial P$
is either one point or a line segment. We consider
polytopes $P$ which have the property:
\begin{equation}\label{one}
|\rho(y,P)| = 1, \forall y \in \partial \pi(P),
\end{equation}
i.e., $\rho(y,P)$ always has only one point for a boundary point $y.$ We will see later (Corollary \ref{lfone}) that any lattice-face polytope $P$ has the property \eqref{one}. The following lemma on polytopes satisfying \eqref{one} will be used in the proof of Theorem \ref{main1}.

\begin{lem}\label{lemone}
If $P = \bigsqcup_{i=1}^k P_i,$ where all $P_i$'s satisfy \eqref{one}, then $P$ satisfies \eqref{one} as well.
Furthermore,  $\Omega(P) = \bigoplus_{i=1}^k \Omega(P_i).$
\end{lem}

\begin{proof}
In Lemma 2.5/(v) of \cite{lattice-face}, we have already shown that if $P$ and the $P_i$'s all satisfy \eqref{one}, then $\Omega(P) = \bigoplus_{i=1}^k \Omega(P_i).$ Therefore, it is enough to show that the condition that all $P_i$'s satisfy \eqref{one} implies that $P$ satisfies \eqref{one}.

For any $y \in \partial \pi(P),$ since $P = \bigsqcup_{i=1}^k P_i,$  we have that $\rho(y, P) = \bigcup_{i=1}^k \rho(y, P_i).$ If $y \in \partial \pi(P_i),$ then $\rho(y, P_i)$ has one element. Otherwise if $y$ is not in $\partial \pi(P_i),$ since $y$ is a boundary point of $\pi(P),$ $y$ cannot be a point in the interior of $P_i,$ so $y \not\in P_i$ and $\rho(y, P_i)$ is the empty set. Hence, in either case, $\rho(y, P_i)$ has finitely many points, for each $i = 1, \dots, k.$ Therefore, $\rho(y, P)$ has finitely many points and cannot be a line segment. Thus, $|\rho(y,P)| = 1.$
\end{proof}









For simplicity, for any set $S \in
\R^d,$ we denote by $\L(S) = S \cap \Z^d$ the set of lattice points
in $S.$ 

\section{A new definition of Lattice-face polytopes}

We first recall the old definition of lattice-face polytopes given in \cite{lattice-face}. 
\subsection*{Old definition (Definition 3.1 in \cite{lattice-face}):} We define {\it lattice-face} polytopes recursively.
We call a $1$-dimensional polytope a {\it lattice-face} polytope if
it is integral.

For $d \ge 2,$ we call a $d$-dimensional polytope $P$ with vertex
set $V$ a {\it lattice-face} polytope if for any $d$-subset $U
\subset V,$
\begin{alist}
\itm $\pi(\conv(U))$ is a lattice-face polytope, and

\itm $\pi(\L(\aff(U))) = \Z^{d-1}.$ In other words, after dropping the last coordinate
of the lattice of $\aff(U),$ we get the $(d-1)$-dimensional lattice.
\end{alist}

With this old definition, it is shown in Lemma 3.3/(v) of \cite{lattice-face} that for any lattice-face $d$-polytope with vertex set $V,$ 
\begin{equation}\label{bad}
\mbox{any $d$-subset $U$ of $V$ forms a $(d-1)$-simplex.}
\end{equation}
This implies that any $(d-2)$-dimensional face only has $(d-1)$ vertices. It is clear that not any rational polytope has this property, e.g., a $4$-dimensional cube. Therefore, with the old definition, the family of lattice-face polytopes does not contain all combinatorial types of rational polytopes. Luckily, we are able to revise the definition such that the restriction \eqref{bad} does not apply.
 
\begin{defn}\label{dflf1} We define {\it lattice-face} polytopes recursively.
We call a $1$-dimensional polytope a {\it lattice-face} polytope if
it is integral.

For $d \ge 2,$ we call a $d$-dimensional polytope $P$ with vertex
set $V$ a {\it lattice-face} polytope if for any subset $U
\subset V$ spanning a $(d-1)$-dimensional affine space,
\begin{alist}
\itm $\pi(\conv(U))$ is a lattice-face polytope, and

\itm $\pi(\L(\aff(U))) = \Z^{d-1}.$ 
\end{alist}
\end{defn}

\begin{rem}\label{dflf2} We have an alternative definition of lattice-face
polytopes, which is equivalent to Definition \ref{dflf1}. Indeed, a
$d$-polytope on a vertex set $V$ is a lattice-face polytope if and
only if  for all $k$ with $0 \le k \le d-1,$
for any subset $U \subset V$ spanning a $k$-dimensional affine space,  $\pi^{(d-k)}(\L(\aff(U))) =
\Z^k.$
\end{rem}

To avoid confusion, from now on, we will call lattice-face polytopes defined under Definition 3.1 in \cite{lattice-face} {\it old lattice-face polytopes}.

Lemma 3.3 in \cite{lattice-face} gives properties of an old lattice-face polytope. All but one of the properties still hold for lattice-face polytopes under the new definition, and the proofs are similar. We state them here without a proof.

\begin{lem}\label{plf}
 Let $P$ be a lattice-face $d$-polytope with vertex set
$V,$ then we have:
\begin{ilist}
\itm $\pi(P)$ is a lattice-face $(d-1)$-polytope.

\itm $m P$ is a lattice-face $d$-polytope, for any positive integer
$m.$

\itm $\pi$ induces a bijection between $\L(NB(P))$ (or $\L(PB(P))$)
and $\L(\pi(P)).$

\itm $\pi(\L(P)) = \L(\pi(P)).$

\itm Let $H$ be a $(d-1)$-dimensional affine space spanned by some subset of $V.$
Then for any lattice point $y \in \Z^{d-1},$ we have that $\rho(y,
H)$ is a lattice point.

\itm $P$ is an integral polytope.
\end{ilist}
\end{lem}

One might ask what is the relation between the family of old lattice-face polytopes and the family of newly defined lattice-face polytopes. We have the following lemma.
\begin{lem}
Every old lattice-face polytope is a lattice-face polytope.
\end{lem}
\begin{proof}
We prove the lemma by induction on $d,$ the dimension of the polytope. If $d=1,$ an old lattice-face polytope is a lattice-face polytope by definition. Suppose $d \ge 2$ and any old lattice-face polytope of dimension smaller than $d$ is a lattice-face polytope. Let $P$ be an old lattice-face $d$-polytope with vertex set $V.$ For any subset $U \subset V$ spanning a $(d-1)$-dimensional affine space, we must have that $|U| \ge d.$ For any $d$-subset $W$ of $U,$ by the definition of old lattice-face polytopes, we have the following:
\begin{itemize}
\itm $\pi(\conv(W))$ is an old lattice-face polytope.
\itm $\pi(\L(\aff(W))) = \Z^{d-1}.$
\end{itemize}
Using \eqref{bad}, $W$ forms a $(d-1)$-simplex. Therefore,
$$\aff(W) = \aff(U) \ \Rightarrow \pi(\L(\aff(U))) = \Z^{d-1}.$$

Let $\overline{V}$ be the vertex set of $\pi(\conv(U)).$ For any $d$-subset $\overline{U} \subset \overline{V},$ there exists a $d$-subset $W$ of $U$ such that $\pi(W) = \overline{U}.$ Then we have that $\pi(\conv(W)) = \conv(\pi(W)) = \conv(\overline{U})$ is an old lattice-face polytope. It is easy to check that if for any $d$-subset $\overline{U} \subset \overline{V},$ we have that $\conv(\overline{U})$ is an old lattice-face polytope, then $\conv(\overline{V})$ is an old lattice-face polytope. Therefore, we conclude that $\pi(\conv(U))$ is an old lattice-face $(d-1)$-polytope.
Thus, by the induction hypothesis, $\pi(\conv(U))$ is a lattice-face $(d-1)$-polytope. Therefore, $P$ is a lattice-face polytope.

\end{proof}

It is easy to check that if $P$ is a $d$-simplex, then $P$ is an old lattice-face polytope if and only if $P$ is a lattice-face polytope. Therefore, the following proposition follows from Theorem 3.6 in \cite{lattice-face}.

\begin{prop}
For any $P$ a lattice-face simplex, the number of lattice points in the nonnegative part of $P$ is equal to the volume of $P:$
$$|\L(\Omega(P))| = \vol(P).$$
\end{prop}

\section{Proof of Theorem \ref{main1}}
We break the proof of Theorem \ref{main1} into the following two propositions.

\begin{prop}\label{validtri}
Let $P$ be a lattice-face $d$-polytope, and $P = \bigsqcup_{i=1}^\ell P_i$ be a triangulation of $P$ without introducing new vertices. Then each $P_i$ is a lattice-face $d$-simplex.
\end{prop}

\begin{prop}\label{tri2form}
For any $d$-polytope $P,$ if $P$ has a triangulation $\bigsqcup_{i=1}^\ell P_i$ consisting of lattice-face $d$-simplices, then $\Omega(P) = \bigoplus_{i=1}^\ell \Omega(P_i).$ Thus,
$$|\L(\Omega(P))| = \vol(P).$$
Furthermore, the Ehrhart polynomial of $P$ is given by \eqref{ques}
$$i(P, m) =  \vol(mP) + i(\pi(P), m) = \sum_{k=0}^d
\vol_k(\pi^{(d-k)}(P)) m^k.$$
\end{prop}

\begin{rem}
Note that we allow triangulation involving new vertices other than vertices of $P$ here. It is implicit in \cite{lattice-face} that if $P$ has a triangulation without new vertices consisting of lattice-face simplices, then the Ehrhart polynomial of $P$ satisfies \eqref{ques}. However, since we allow introducing new vertices to form a triangulation, Proposition \ref{tri2form} takes care of more cases. See Example \ref{tri2}. 
\end{rem}

It is clear that Theorem \ref{main1} follows from Proposition \ref{validtri} and Proposition \ref{tri2form}. Proposition \ref{validtri} can be proved directly by checking the definition of lattice-face polytopes, so we will only give the proof of Proposition \ref{tri2form}. Before that, we first prove the following lemma.

\begin{lem}\label{lfsone}
Any lattice-face simplex satisfies \eqref{one}.
\end{lem}
\begin{proof}
Suppose $P$ is a lattice-face $d$-simplex that does not satisfy \eqref{one}. There exists $y \in \partial \pi(P)$ such that $\rho(y, P)$ is a line segment. Let $F$ be a facet of $P$ that contains $\rho(y, P)$ and $U$ be the vertex set of $F.$ Then $\aff(U)$ is a $(d-1)$-dimensional affine space in $\Z^d$ that contains a line $L = \aff(\rho(y,P)).$ Because $\pi$ sends $L$ to a point $y,$ the dimension of $\pi(\aff(U))$ is smaller than $d-1.$ Hence, $\pi(\L(\aff(U))) \neq \Z^{d-1}.$ This contradicts part b) in Definition \ref{dflf1}.
\end{proof}

The following corollary follows from Lemma \ref{lfsone}, Lemma \ref{lemone} and Proposition \ref{validtri}.
\begin{cor}\label{lfone}
Any lattice-face polytope satisfies \eqref{one}.
\end{cor}

\begin{proof}[Proof of Proposition \ref{tri2form}]
By Lemma \ref{lemone} and Lemma \ref{lfsone}, we immediately have $\Omega(P) = \bigoplus_{i=1}^\ell \Omega(P_i).$ Thus, 
$$|\L(\Omega(P))| = |\bigoplus_{i=1}^\ell \L(\Omega(P_i))| = \sum_{i=1}^\ell |\L(\Omega(P_i))| = \sum_{i=1}^\ell \vol(P_i) = \vol(P).$$
Hence, using this and Lemma \ref{plf}/(iii), $$|\L(P)| = |\L(\Omega(P))| + |\L(NB(P))| = \vol(P) + |\L(\pi(P))|.$$
Note that for any positive integer $m,$ the dilation $m P$ has a triangulation $\bigsqcup_{i=1}^\ell mP_i$ where each $m P_i$ is still a lattice-face $d$-simplex by Lemma \ref{plf}/(ii). Therefore,
\begin{equation}\label{ehrvol}
i(P,m) = |\L(mP)| = \vol(mP) + |\L(\pi(mP))| = \vol(mP) + i(\pi(P), m).
\end{equation}
Let $\{F_1, \dots, F_{\ell'} \}$ be the set of facets of $P_1, \dots, P_\ell$ that are contained in the negative boundary $NB(P)$ of $P.$ Then we have $NB(P) = \bigsqcup_{i=1}^{\ell'} F_i$ and 
$$\pi(P) = \pi(NB(P)) =  \bigsqcup_{i=1}^\ell \pi(F_i).$$
One checks that $\pi(F_i)$'s are lattice-face $(d-1)$-simplices. Therefore, we can replace $i(\pi(P), m)$ in \eqref{ehrvol} with $\vol_{d-1}(m \ \pi(P)) + i(\pi^{(2)}(P),m):$
$$i(P,m) = \vol(mP) + i(\pi(P), m) = \vol(mP) + \vol_{d-1}(m \ \pi(P)) + i(\pi^{(2)}(P),m).$$
Applying this argument inductively, we obtain
$$i(P,m) = \sum_{k=0}^d \vol_{k}(m \ \pi^{(d-k)}(P)) = \sum_{k=0}^d \vol_{k}(\pi^{(d-k)}(P)) m^k.$$
\end{proof}

As we mentioned in the introduction, Proposition \ref{tri2form} provides a larger family of polytopes than the family of lattice-face polytopes which still have Ehrhart polynomials satisfying \eqref{ques}. We will finish this section with two examples of polytopes $P$ where $P$ is not a lattice-face polytope, but by using Proposition \ref{tri2form}, we still have that $i(P,m)$ satisfies \eqref{ques}.
\begin{ex}
\label{tri1}
Let $P$ be the polygon with vertices $\{(0,0), (2,0), (1,1), (3,1) \}.$ One can check that $P$ is not a lattice-face polytope because $\pi(\L(\aff(\{(0,0),(3,1)\}))) = \{ 3z \ | \ z \in Z\} \neq \Z.$ However, $P$ has a triangulation without introducing new vertices $P = \conv(\{(0,0),(2,0),(1,1)\}) \sqcup \conv(\{(2,0),(1,1),(3,1)\}),$ where both triangles are lattice-face simplices. Therefore, by Proposition \ref{tri2form}, 
$$i(P,m) = \sum_{k=0}^2 \vol_{k}(\pi^{(2-k)}(P)) m^k =2m^2 + 3m +1.$$
\end{ex}

\begin{ex}
\label{tri2}
Let $P$ be the polygon with vertices $\{v_1=(0,0), v_2=(4,0), v_3=(3,5), v_4=(1,5) \}.$ One can check that $P$ is not a lattice-face polytope because $\pi(\L(\aff(\{v_1, v_3\}))) = \{ 3z \ | \ z \in Z\} \neq \Z.$ There are only two possible triangulations of $P$ without introducing new vertices, but neither of them is one consisting of lattice-face simplices. However, if we introduce a new vertex $v_5 = (2,4),$ we can obtain the triangulation $P = \conv(\{v_1, v_2, v_5\}) \sqcup \conv(\{(v_2, v_3, v_5\}) \sqcup \conv(\{(v_3, v_4, v_5\}) \sqcup \conv(\{(v_1, v_4, v_5\}),$ where all triangles are lattice-face simplices. Therefore, by Proposition \ref{tri2form}, 
$$i(P,m) = \sum_{k=0}^2 \vol_{k}(\pi^{(2-k)}(P)) m^k =15m^2 + 4m +1.$$
\end{ex}

\section{Polytopes in $\pi$-general position and the proof of Theorem \ref{main2}}
For any linear transformation $\phi: \R^d \to \R^d,$ we can associate a $d \times d$ matrix $M$ to $\phi,$ such that $\phi(\x) = M(x).$ Therefore, when we describe a linear transformation, we often just describe it by its associated matrix. 

We denote by  $\diag(A_1, \dots, A_k)$ the block diagonal matrix with square matrices $A_1, \dots, A_k$ on the diagonal. In particular,  $\diag(c_1, \dots, c_d)$ denotes the $d \times d$ diagonal matrix with diagonal entries $c_1, \dots, c_d.$

\begin{defn} We say that a finite set $V \subset \R^d$ is in {\it
$\pi$-general position} if $\aff(V) = \R^d$ and for any $k: 0 \le k \le d-1,$ and any subset $U \subset V$, we have that 
\begin{equation}\label{gpprop}
\mbox{if $\aff(U)$ is $k$-dimensional, then $\pi^{d-k}(\aff(U))$ is $k$-dimensional.}
\end{equation}

We say that a $d$-polytope $P$ in {\it $\pi$-general position} if its vertex set is {\it in $\pi$-general position}.
\end{defn}

\begin{rem}\label{d-1}
We can understand property \eqref{gpprop} in a more algebraic way. In particular, when $k = d-1,$ a $(d-1)$-dimensional affine space $H$ has the property that $\pi(H)$ is $(d-1)$-dimensional if and only if the last coordinate of the normal vector of $H$ is nonzero.
\end{rem}

By the alternative definition of lattice-face polytopes in Remark
\ref{dflf2}, it's easy to see that a lattice-face polytope is a
polytope in $\pi$-general position. Therefore, we use rational polytopes in $\pi$-general position as the bridge to connect lattice-face polytopes and general rational polytopes. In fact, the proof of Theorem \ref{main2} follows from the following two propositions.

\begin{prop}\label{reg2gp}
For any finite set $V \subset \Q^d$ with $\aff(V) = \R^d,$ there exists an invertible linear transformation $\phi$ associated to an upper triangular matrix with integer entries and $1$'s on the diagonal such that $\phi(V)$ is a finite set of rational points and is in $\pi$-general position. 

In particular, if $V$ is the vertex set of a rational $d$-polytope, then $\phi(P)$ is a rational polytope in $\pi$-general position.
\end{prop}

\begin{prop}\label{gp2lf}
Suppose $V \subset \Q^d$ is a finite set in $\pi$-general position. Then there exist nonzero integers $c_1, \dots, c_d,$ such that for any $U \subset V$ with $\aff(U) = \R^d,$ we have that $\phi(\conv(U))$ is a lattice-face polytope, where $\phi$ is the invertible linear transformation associated with the diagonal matrix $\diag(c_1, \dots, c_d)$.

In particular, if $V$ is the vertex set of $P$ a rational $d$-polytope in $\pi$-general position, then $\phi(P)$ is a lattice-face polytope.
\end{prop}

We prove these propositions in two subsections below. Each proof is preceded by a pair of lemmas. It is easy to check that the operator $\aff$ commutes with any linear transformation. We will use this fact often in the proofs.

\subsection{Proof of Proposition \ref{reg2gp}}
\begin{lem}\label{gpd-1}
Suppose $V$ is a finite set in $\R^d$ such that $\aff(V) = \R^d$ and for $k = d-1$ and any $U \subset V$, \eqref{gpprop} is satisfied. Then for any $k < d-1,$ and any $U \subset V$ spanning a $k$-dimensional affine space, we have that $\pi(\aff(U))$ is $k$-dimensional.
\end{lem}
\begin{proof}
Because $\aff(V) = \R^d,$ there exist $v_1, \dots, v_{d-1-k} \in V \setminus U,$ such that $\aff(U \cup \{ v_1, \dots, v_{d-k} \})$ is $(d-1)$-dimensional. Let $\tilde{U} = U \cup \{ v_1, \dots, v_{d-1-k}\}.$ Then we have that $\aff(\pi(\tilde{U})) = \pi(\aff(\tilde{U}))$ has dimension $d-1.$ However, the number of elements in $\pi(\tilde{U}) \setminus \pi(U)$ is at most $d-1-k.$ Hence, the dimension of $\pi(\aff(U)) = \aff(\pi(U))$ is at least $k.$ Since the dimension of $\aff(U)$ is $k,$ we must have that $\pi(\aff(U))$ is $k$-dimensional.  
\end{proof}

\begin{lem}\label{findM}
Suppose $d \ge 2$ and $H$ is a $(d-1)$-dimensional affine space in $\R^d$ with normal vector $\n.$ 
\begin{ilist}
\itm Suppose $\v = (v_1, \dots, v_{d})^T$ is a vector that is not in the null space of $\n^T,$ i.e., $\v \cdot \n \neq 0,$ and $v_d = 1.$ Let $M_{\v}$ be the the $d \times d$ upper triangular matrix with $1$'s on the diagonal, $-v_i$ the $(i, d)$-entry for $1 \le i \le d,$ and $0$'s elsewhere. Then $\pi(M_\v(H))$ is $(d-1)$-dimensional.
\itm Suppose $A$ is a $(d-1) \times (d-1)$ invertible matrix. Let $M_A$ be the $d \times d$ block diagonal matrix $\diag(A, 1).$ Then $\pi(H)$ is $(d-1)$-dimensional if and only if $\pi(M_A(H))$ is $(d-1)$-dimensional.
\end{ilist}
\end{lem}
\begin{proof}
There exists $a \in \R,$ such that $H = \{ \x \ | \ \n \cdot \x = a \}.$ For any invertible matrix $M,$ it is easy to check that $M(H)$ is the $(d-1)$-dimensional affine space $\{ \x \ | \ ((M^{-1})^T \n) \cdot \x = a \}.$ Hence, the normal vector of $M(H)$ is $(M^{-1})^T \n.$
\begin{ilist}
\itm By Remark \ref{d-1}, $\pi(M_\v(H))$ is $(d-1)$-dimensional if and only if the last coordinate of the normal vector of $M_\v(H)$ is nonzero. However, the normal vector of $M_\v(H)$ is $(M_{\v}^{-1})^T \n.$ One can verify that $(M_{\v}^{-1})^T$ is the $d \times d$ lower triangular matrix matrix with $1$'s on the diagonal, $v_i$ the $(d,i)$-entry for $1 \le i \le d,$ and $0$'s elsewhere. In particular, the last row of $(M_{\v}^{-1})^T$ is $\v^T.$ Therefore, the last coordinate of $(M_{\v}^{-1})^T \n$ is $\v^T \n = \v \cdot \n \neq 0.$

\itm $(M_A^{-1})^T$ is the block diagonal matrix $\diag((A^{-1})^T, 1).$ One checks that the last coordinate of $(M_A^{-1})^T \cdot \n,$ the normal vector of $M_A(H)$, is exactly the last coordinate of $\n,$ the normal vector of $H.$ Therefore, by Remark \ref{d-1}, $\pi(H)$ is $(d-1)$-dimensional if and only if $\pi(M_A(H))$ is $(d-1)$-dimensional.
\end{ilist}
\end{proof}

\begin{proof}[Proof of Proposition \ref{reg2gp}]
We prove the proposition by induction on $d.$ When $d=1,$ any finite set $V \in \Q^d$ is in $\pi$-general position by definition. Assume $d \ge 2$ and the proposition holds when the dimension is smaller than d. Since $V$ is a finite set, there are finitely many $(d-1)$-dimensional affine spaces spanned by subsets of $V.$ Suppose they are $H_1, \dots, H_\ell$ with normal vectors $\n_1, \dots, \n_\ell,$ respectively. For each $i: 1 \le i \le \ell,$ the null space of $\n_i^T$ is a $(d-1)$-dimensional subspace in $\R^d.$ Let $H_0$ the the set of points in $\R^d$ with the last coordinate equal to $1$. The set $\H = \{ H_i \cap H_0 \neq \emptyset : 1 \le i \le \ell \}$ is a finite set of $(d-2)$-dimensional affine spaces inside $H_0.$ One sees easily that the complement (with respect to $H_0$) of the union over $\H$ contains lattice points. Therefore,
there exists a vector $\v \in \Z^d \cap H_0$ such that $\v \not\in H_i,$ for each $i: 1 \le i \le \ell.$ We pick such a $\v$, then we have that $\v \cdot \n_i \neq 0$ for each i, and the last coordinate of $\v$ is $1$. Let $\phi_0$ be the invertible linear transformation associated to $M_\v,$ where $M_\v$ is the matrix described in Lemma \ref{findM}/(i). By Lemma \ref{findM}/(i), we have that $\pi(\phi_0(H_i))$ is $(d-1)$-dimensional for any $i: 1 \le i \le \ell.$

Let $\overline{V} = \phi_0(V).$ Because $\pi(\overline{V}) \subset \Q^{d-1},$ by the induction hypothesis, there exists an upper triangular matrix $A$ with integer entries and $1$'s on the diagonal, such that $A(\pi(\overline{V}))$ is a finite set of rational points and is in $\pi$-general position. Let $\psi$ be the invertible linear transformation associated to the block diagonal matrix $\diag(A, 1).$ It is easy to see that $A \circ \pi = \pi \circ \psi.$ Let $\phi = \psi \circ \phi_0.$ Because both $\psi$ and $\phi_0$ are upper triangular matrix with integer entries and $1$'s on the diagonal, $\phi$ is such a matrix too. It is clear that $\phi(V)  = \psi(\overline{V})$ is a finite set of rational points. We will show that $\phi(V) = \psi(\overline{V})$ is in $\pi$-general position.

Since both of $\psi$ and $\phi_0$ are invertible, $\phi$ is invertible as well. Therefore, $\aff(\phi(V)) = \aff(V) = \R^d.$ We only need to show that for any subset $U \subset \phi(V),$ \eqref{gpprop} holds for any $k: 0 \le k \le d-1.$ 


Suppose $\aff(U)$ is $(d-1)$-dimensional. Since $\phi$ is invertible, we have that $\aff(\phi^{-1}(U))$ is $(d-1)$-dimensional. We know that $\aff(\phi^{-1}(U)) = H_i,$ for some $i = 1, \dots, \ell.$ As discussed above, $\pi(\phi_0(H_i))$ is $(d-1)$-dimensional. However, by Lemma \ref{findM}/(ii), we also have $\pi(\phi_0(H_i))$ is $(d-1)$-dimensional if and only if $\pi(\psi(\phi_0(H_i)))$ is $(d-1)$-dimensional. Thus, we conclude that $\pi(\aff(U)) = \pi(\aff(\phi(\phi^{-1}(U)))) = \pi(\phi(\aff(\phi^{-1}(U)))) = \pi(\phi(H_i)) = \pi(\psi(\phi_0(H_i)))$ is $(d-1)$-dimensional.

Suppose $\aff(U)$ is $k$-dimensional, where $0 \le k < d-2.$ Let $\overline{U} =\psi^{-1}(U).$ Since $\psi$ is invertible, we have that $\aff(\overline{U})$ is $k$-dimensional. Note that $\overline{U}$ is a subset of $\overline{V} = \phi_0(V).$ However, by the construction of $\phi_0,$ we know that $\overline{V}$ is a set satisfying the hypothesis in Lemma \ref{gpd-1}. Therefore, by Lemma \ref{gpd-1}, we have that $\aff(\pi(\overline{U})) = \pi(\aff(\overline{U}))$ is $k$-dimensional. Because $A$ is invertible, $\aff(A(\pi(\overline{U})))$ is $k$-dimensional. But $A(\pi(\overline{U}))$ is a subset of $A(\pi(\overline{V})),$ which is in $\pi$-general position (in $\R^{d-1}$). Thus, $\pi^{(d-1-k)}(\aff(A(\pi(\overline{U}))))$ is $k$-dimensional. Therefore, the dimension of $\pi^{(d-k)}(\aff(U)) = \pi^{(d-1-k)}(\aff(\pi(U))) = \pi^{(d-1-k)}(\aff(\pi(\psi(\overline{U})))) = \pi^{(d-1-k)}(\aff(A(\pi(\overline{U}))))$ is $k.$
\end{proof}

\subsection{Proof of Proposition \ref{gp2lf}}
\begin{lem}\label{findc}
Suppose $U \subset \Q^d$ such that both $\aff(U)$ and $\pi(\aff(U))$ are $(d-1)$-dimensional. Then there exists a nonzero integer $c_U$ such that for any $c$ a nonzero multiple of $c_U,$ 
we have that $$\pi(\L(\aff(\phi_c(U)))) = \Z^{d-1},$$
where $\phi_c$ is the invertible linear transformation associated with the diagonal matrix $\diag(1, \dots, 1, c)$.
\end{lem}
\begin{proof}
Since $U$ is a set of points in $\Q^d$ and $\aff(U)$ is $(d-1)$-dimensional, $\aff(U)$ can be described by a linear equation:  $\alpha_1 x_1 + \cdots \alpha_d x_d = a,$ for some integers $\alpha_1, \dots, \alpha_d, a.$ Because $\pi(\aff(U))$ is $(d-1)$-dimensional, by Remark \ref{d-1}, we must have that $\alpha_d \neq 0.$ Let $c_U = \alpha_d.$ If $c = k c_U,$ for some nonzero $k \in \Z,$ then $\aff(\phi_c(U))$ is the $(d-1)$-dimensional affine space defined by the equation $\alpha_1 x_1 + \cdots \alpha_{d-1} x_{d-1} + \frac{1}{k} x_d = a.$ For any lattice point $\y = (y_1, \dots, y_{d-1}) \in \Z^{d-1},$ the intersection of $\aff(\phi_c(U))$ with the inverse image of $\y$ under $\pi$ is the point $(y_1, \dots, y_{d-1}, k (a - (\alpha_1 x_1 + \cdots \alpha_{d-1} x_{d-1} ))),$ which is a lattice point. Therefore, $\pi(\L(\aff(\phi_c(U)))) = \Z^{d-1}.$
\end{proof}

\begin{lem}\label{piV}
Suppose $d \ge 2.$ For any finite set $V \subset \R^d$ in $\pi$-, the finite set $\pi(V) \subset \R^{d-1}$ is in $\pi$-general position.
\end{lem}

\begin{proof}
Because $\aff(V) = \R^d,$ there exists $U \subset V$ such that $\aff(U)$ is $(d-1)$-dimensional. Thus, $\pi(\aff(U))$ is $(d-1)$-dimensional. Hence, $\aff(\pi(V)) = \pi(\aff(V)) \supset \pi(\aff(U))$ has to be $(d-1)$-dimensional. 

For any $k: 0 \le k \le d-2,$ let $U'$ be a subset of $\pi(V)$ that spans a $k$-dimensional affine space. There exists $U \subset V$ such that $\pi(U) = U'.$ We need to show that $\pi^{(d-1-k)}(\aff(U')) = \pi^{(d-1-k)}(\aff(\pi(U))) = \pi^{(d-k)}(\aff(U))$ is $k$-dimensional.  Since $V$ is in $\pi$-general position, it is enough to show that $\aff(U)$ is $k$-dimensional. Suppose $\dim(\aff(U)) \neq  k.$ Given that $\aff(U') = \aff(\pi(U)) = \pi(\aff(U))$ is $k$-dimensional, we must have that $\dim(\aff(U)) \ge k+1.$ There exists $W \subset U$ such that $\dim(\aff(W)) = k+1.$ Note that $k+1 \le d-1.$ Because $V$ is in $\pi$-general position, $\pi^{(d-(k+1))}(\aff(W))$ is $(k+1)$-dimensional. This implies that $\dim(\pi^{(i)}(\aff(W))) = k+1,$ for any $i$ with $0 \le i \le d-(k+1).$ In particular, $\pi(\aff(W)) = \aff(\pi(W))$ has dimension $k+1.$ However, $\aff(U') = \aff(\pi(U)) \supset \aff(\pi(W)),$ but $\dim(\aff(U')) = k < \aff(\pi(W)).$ This is a contradiction. 

\end{proof}

\begin{proof}[Proof of Proposition \ref{gp2lf}]
We prove the proposition by induction on $d.$ When $d=1,$ $V = \{v_1, \dots, v_n\}$ where each $v_i$ is a rational number, thus can be written as $v_i = \frac{p_i}{q_i}$ for some integers $p_i$ and $q_i \neq 0.$ Let $\phi = \prod_{i=1}^n q_i.$ Clearly, $\phi(V)$ is a set of integers. For any $U \subset V$ with $\aff(U) = \R,$ we have that $\phi(\conv(U))$ is an integral polytope, thus is a lattice-face polytope.

Assume $d \ge 2$ and the lemma is true when the dimension (of $\aff(V)$) is smaller than $d.$ By Lemma \ref{piV}, $\pi(V) \subset \Q^{d-1}$ is in $\pi$-general position in $\R^{d-1}$. Thus, there exists nonzero integers $c_1, \dots, c_{d-1},$ such that  for any $U' \subset \pi(V)$ with $\aff(U') = \R^{d-1},$ we have that $\conv(\psi(U'))$ is a lattice-face polytope, where $\psi=\diag(c_1, \dots, c_{d-1})$. 
Let $\nTilde{\psi} = \diag(c_1, \dots, c_{d-1}, 1)$ be the linear transformation in $\R^d$ corresponding to $\psi.$ Let $U \subset V$ such that $\aff(U)$ is $(d-1)$-dimensional. Because $V$ is in $\pi$-general position, we have that $\pi(\aff(U))$ is $(d-1)$-dimensional. Since $\psi$ and $\nTilde{\psi}$ are invertible and $\pi \circ \tilde{\psi} = \psi \circ \pi$, $\aff(\nTilde{\psi}(U)) = \nTilde{\psi}(\aff(U))$ and $\pi(\aff(\nTilde{\psi}(U))) = \psi(\pi(\aff(U)))$ are $(d-1)$-dimensional. Note that $\nTilde{\psi}(U)$ are still in $\Q^d.$ Thus, by Lemma \ref{findc}, there exists a nonzero integer $c_U$ such that for any $c$ a nonzero multiple of $c_U,$ we have that $\pi(\L(\aff(\phi_c(\nTilde{\psi}(U))))) = \Z^{d-1},$
where $\phi_c = \diag(1, \dots, 1, c)$. Let 
$$c_d = \prod_{\mbox{\small $U \subset V: \aff(U)$  is $(d-1)$-dimensional}} c_U.$$
We claim the linear transformation $$\phi = \diag(1, \dots, 1, c_d) \circ \nTilde{\psi} = \diag(c_1, \dots, c_{d-1}, c_d)$$ has the desired property.
For any $U \subset V$ with $\aff(U) = \R^d$, we need to check that $P_U = \phi(\conv(U)) = \conv(\phi(U))$ is a lattice-face polytope. It is enough to check the case when $\phi(U)$ is the vertex set of $P_U.$ We will show this by checking the definition of lattice-face polytopes. For any subset $W \subset \phi(U)$ spanning a $(d-1)$ dimensional affine space, we need to show that 
$$\mbox{ a) $\pi(\conv(W))$ is a lattice-face polytope, and b) $\pi(\L(\aff(W))) = \Z^{d-1}.$}$$
Since $\phi$ is an invertible linear transformation,  we have that $\overline{W}:= \phi^{-1}(W) \subset U \subset V$ and $\aff(\overline{W})$ is $(d-1)$-dimensional. As we mentioned before, because $V$ is in $\pi$-general position, $\pi(\aff(\overline{W}))$ is $(d-1)$-dimensional.
\begin{alist}
\itm It is clear that $\pi \circ \phi = \psi \circ \pi.$ Thus, $\pi(W) = \pi (\phi(\overline{W})) = \psi(\pi(\overline{W})).$ Therefore, $\pi(\conv(W)) = \conv(\pi(W)) = \conv(\psi(\pi(\overline{W}))).$ However, since $\pi(\overline{W}) \subset \pi(V),$ $\conv(\psi(\pi(\overline{W})))$ is a lattice-face polytope.
\itm Since $c_d$ is a nonzero multiple of $c_{\overline{W}},$ we have that $$\pi(\L(\aff(W))) = \pi(\L(\aff(\phi(\overline{W})))) = \pi(\L(\aff(\phi_{c_d}(\nTilde{\psi}(\overline{W}))))) = \Z^{d-1}.$$
\end{alist}
\end{proof}











\bibliographystyle{amsplain}
\bibliography{gen}

\end{document}

%% file: headerstemplate.tex
\usepackage{amssymb,euscript,amsmath, mathrsfs}
\usepackage[dvips]{graphicx}
\usepackage[dvips]{color}
\usepackage{epsfig}

\newcounter{ENUM}
\newcommand{\itm}{\item}
\newenvironment{ilist}{\renewcommand{\theENUM}{\roman{ENUM}}\renewcommand{\itm}{\addtocounter{ENUM}{1}\item[(\theENUM)]}\begin{itemize}\setcounter{ENUM}{0}}{\end{itemize}}
\newenvironment{alist}[1][0]{\renewcommand{\theENUM}{\alph{ENUM}}\renewcommand{\itm}{\addtocounter{ENUM}{1}\item[\theENUM)]}\begin{itemize}\setcounter{ENUM}{#1}}{\end{itemize}}

\newcommand{\nTilde}[1]{\widetilde{#1}}                    

\def\x{{\mathbf x}}
\def\y{{\mathbf y}}

\def\n{{\mathbf n}}
\def\v{{\mathbf v}}

\def\1{{\mathbf 1}}
\def\diag{\operatorname{diag}}

\def\Z{{\mathbb Z}}
\def\N{{\mathbb N}}

\def\Q{{\mathbb Q}}
\def\R{{\mathbb R}}

\def\H{{\mathbb H}}

\def\L{{\mathcal L}}

\def\sgn{\operatorname{sign}}

\def\conv{\mathrm{conv}}
\def\aff{\mathrm{aff}}
\def\vol{\mathrm{Vol}}

\newtheorem{thm}{Theorem}[section]
\newtheorem{prop}[thm]{Proposition}
\newtheorem{lem}[thm]{Lemma}
\newtheorem{cor}[thm]{Corollary}

\theoremstyle{definition}
\newtheorem{defn}[thm]{Definition}

\newtheorem{ex}[thm]{Example}

\theoremstyle{remark}

\newtheorem{rem}[thm]{Remark}

\numberwithin{equation}{section}